\documentclass[a4paper,11pt]{article}
\usepackage[scale={0.8,0.9},centering,includeheadfoot]{geometry}
\usepackage{amssymb,amsmath, graphicx,epsfig,amscd,array,color,url,subfigure}
\usepackage{amsthm}
\usepackage[round]{natbib}
\newtheorem{theorem}{Theorem}

\newtheorem{proposition}{Proposition}

 \newcommand{\norm}[1]{\lVert#1\rVert}

 \def\mm#1{\ensuremath{\boldsymbol{#1}}} 

\begin{document}

\title{Think continuous: Markovian  Gaussian models in spatial statistics }
\author{Daniel Simpson\footnote{Corresponding author. Email: \texttt{Daniel.Simpson@math.ntnu.no}}, Finn Lindgren \& H\aa{}vard
    Rue\\
    Department of Mathematical Sciences\\
    Norwegian University of Science and Technology\\
    N-7491 Trondheim, Norway} \date{\today}

\maketitle

\begin{abstract}
Gaussian Markov random fields (GMRFs) are frequently used as computationally efficient models in spatial statistics.  Unfortunately, it has traditionally been difficult to link GMRFs with the more traditional Gaussian random field models as the  Markov property is difficult to deploy in continuous space.  Following the pioneering work of \citet{Lindgren2011}, we expound on the link between Markovian Gaussian random fields and GMRFs.  In particular, we discuss the theoretical and practical aspects of fast computation with continuously specified Markovian Gaussian random fields, as well as the clear advantages they offer in terms of clear, parsimonious and interpretable models of anisotropy and non-stationarity. 
\end{abstract}

\section{Introduction}

From a practical viewpoint, the primary difficulty with spatial Gaussian models in applied statistics is dimension, which typically scales with the number of observations.  Computationally speaking, this is a disaster!  It is, however, not a disaster unique to spatial statistics.  Time series models, for example, can suffer from the same problems.  In the temporal case, the ballooning dimensionality is typically tamed by adding a conditional independence, or \emph{Markovian}, structure to the model.  The key advantage of the Markov property for time series models is that the computational burden then grows only linearly (rather than cubically) in the dimension, which makes inference on these models feasible for long time series.  

Despite its success in time series modelling, the Markov property has had a less exalted role in spatial statistics. Almost all instances where the Markov property has been used in spatial modelling has been in the form of Markov random fields defined over a set of discrete locations connected by a graph.  The most common Markov random field  models are \emph{Gaussian} Markov random fields (GMRFs), in which the value of the random field at the nodes is jointly Gaussian \citep{book80}.  GMRFs are typically written as $$
\mm{x} \sim N(\mm{\mu}, \mm{Q}^{-1}), 
$$ where $\mm{Q}$ is the \emph{precision matrix} and the Markov property is equivalent to requiring that $\mm{Q}$ is sparse, that is $Q_{ij} = 0 $ if{f} $x_i$ and $x_j$ are conditionally independent  \citep{book80}.

As problems in spatial statistics are usually concerned with inferring a spatially continuous effect over a domain of interest, it is difficult to directly apply the fundamentally discrete GMRFs.  For this reason, it is commonly stated that there are two essential fields in spatial statistics: the one that uses GMRFs and the one that uses continuously indexed Gaussian random fields.  In a recent read paper, \citet{Lindgren2011} showed that these two approaches are not distinct.  By carefully utilising the continuous space Markov property, it is possible to construct Gaussian random fields for which all quantities of interest can be computed using GMRFs!

The most exciting aspect of the Markovian models of \citet{Lindgren2011} is their flexibility.  There is no barrier---conceptual or computational---to extending them to construct non-stationary, anisotropic Gaussian random fields. Furthermore, it is even possible to construct them on the sphere and other manifolds.  In fact, \citet{SimpsonLGCP} showed that there is essentially no computational difference between inferring a log-Gaussian Cox process on a rectangular observation window and inferring one on a non-convex, multiply connected region on the sphere!  This type of flexibility is not found in any other method for constructing Gaussian random field models.

In this paper we carefully review the connections between GMRFs,
Gaussian random fields, the spatial Markov property and deterministic approximation theory.  It is hoped that this will give the interested reader some insight into the theory and practice of Markovian Gaussian random fields.  In Section \ref{sec:GMRF} we briefly review the practical computational properties of GMRFs.  Section \ref{sec:continuous} we take a detailed tour of the theory of Markovian  Gaussian random fields.  We begin with a discussion of the spatial Markov property and show how it naturally leads to differential operators.  We then present a practical method for approximating Markovian Gaussian random fields and discuss what is mean by a continuous approximation.  In particular, we show that deterministic approximation theory can provide essential insights into the behaviour of these approximations.  We then discuss some practical issues with choosing sets of basis functions before discussing extensions of the models.  Finally we mention some further extensions of the method.

\section{Practical computing with Gaussian Markov random fields} \label{sec:GMRF}

Gaussian Markov random fields possess two pleasant properties that make them useful for spatial problems: they facilitate fast computation for large problems, and they are quite stable with respect to conditioning.  In this section we will explore these two  properties in the context of spatial statistics.

\subsection{Fast computations with Gaussian Markov random fields}
As in the temporal setting, the Markovian property allows for fast computation of samples, likelihoods and other quantities of interest \citep{book80}.    This allows the investigation of much larger models than would be available using general multivariate Gaussian models. The situation is not, however, as good as it is in the one dimensional case, where all of these quantities can be computed using $\mathcal{O}(n)$ operations, where $n$ is the dimension of the GMRF.  Instead, for the two dimensional spatial models, samples and likelihoods can be computed in $\mathcal{O}(n^{3/2})$ operations, which is still a significant saving on the $\mathcal{O}(n^3)$ operations required for a general Gaussian model.  A quick order calculation shows that computing a sample from an $\tilde{n}$--dimensional Gaussian random vector without any special structure takes the same amount of time as computing a sample from GMRF of dimension $n = \tilde{n}^2$!

The key object when computing with GMRFs is the Cholesky decomposition $\mm{Q} = \mm{LL}^T$, where $\mm{L}$ is a lower triangular matrix.  When $\mm{Q}$ is sparse, its Cholesky decomposition can be computed very efficiently \citep[see, for instance,][]{book102}.  Once the Cholesky triangle has been computed, it is easy to show that $\mm{x} = \mm{\mu} + \mm{L}^{-T}\mm{z}$ is a sample from the GMRF $\mm{x} \sim N(\mm{\mu},\mm{Q}^{-1})$ where $\mm{z} \sim N(\mm{0},\mm{I})$.  Similarly, the log density for a GMRF can be computed as $$
\log \pi(\mm{x}) = -\frac{n}{2} \log(2\pi) + \sum_{i=1}^n \log L_{ii} - \frac{1}{2} (\mm{x} - \mm{\mu})^T \mm{Q}(\mm{x} - \mm{\mu}),
$$ where $L_{ii}$ is the $i$th diagonal element of $\mm{L}$ and the inner product $ (\mm{x} - \mm{\mu})^T \mm{Q}(\mm{x} - \mm{\mu})$ can be computed in $\mathcal{O}(n)$ calculations using the sparsity of $\mm{Q}$.  It is also possible to use the Cholesky triangle $\mm{L}$ to compute $\operatorname{diag}(\mm{Q}^{-1})$, which are the marginal variances of the GMRF \citep{art375}.

Furthermore, it is possible to sample from $\mm{x}$ conditioned on a \emph{small number}  of linear constraints $\mm{x} | \mm{Bx}=\mm{b}$, where $\mm{B}\in \mathbb{R}^{k\times n}$ is usually a dense matrix and the number of constraints,  $k$, is very small.  This occurs, for instance, when the GMRF is constrained to sum to zero.  However, if one wishes to sample conditional on data, which usually corresponds to a \emph{large number} of linear constraints, the methods in the next section are almost always significantly more efficient.    While direct calculation of the conditional density is possible,  when $\mm{B}$ is a dense matrix conditioning destroys the Markov structure of the problem.  It is still, however, possible to sample efficiently using a technique known as \emph{conditioning by Kriging} \citep{book80}, whereby an unconditional sample $\mm{x}$ is drawn from $N(\mm{\mu},\mm{Q}^{-1})$ and then corrected using the equation $$
\mm{x}^* = \mm{x} - \mm{Q}^{-1}\mm{B}^T(\mm{BQ}^{-1}\mm{B}^T)^{-1}(\mm{Bx} - \mm{b}).
$$ When $k $ is small, the conditioning by Kriging update can be computed efficiently from the Cholesky factorisation.  We reiterate, however, that when  there are a large number of constraints, the conditioning by Kriging method will be inefficient, and, if $\mm{B}$ is sparse (as is the case when conditioning on data), it is usually better use the methods in the next subsection.   

An alternative method for conditional sampling can be constructed by noting that the conditioning by Kriging update $\mm{x}^* = \mm{x} - \mm{\delta x}$ can be computed by solving the augmented system
\begin{equation} \label{kkt}
\begin{pmatrix} {\mm{Q}} & \mm{B}^T\\ \mm{B}&\mm{0}\end{pmatrix}\begin{pmatrix}\mm{\delta x}\\ \mm{y}\end{pmatrix} = \begin{pmatrix}\mm{0}\\ \mm{Bx} - \mm{b}\end{pmatrix},
\end{equation}
where $\mm{y}$ is an auxiliary variable.  A standard application of Sylvester's inertia theorem \citep[Theorem 8.1.17 in][]{book2} shows that the matrix in \eqref{kkt} is not positive definite, however the system can still be solved using ordinary direct (or iterative) methods \citep{Simpson2008b}.  The augmented system \eqref{kkt} reveals the geometric structure of conditioning by Kriging: augmented systems of equations of this form arise from the Karush-Kuhn-Tucker conditions in constrained optimisation  \citep{benzi_golub}.  In particular, the conditional sample solves the minimisation problem 
\begin{align*}
\mm{x}^* = &\arg \min_{\mm{x}^* \in \mathbb{R}^n} \frac{1}{2}\norm{\mm{x}^* - \mm{x}}_Q^2 \\
&\text{subject to } \mm{Bx}^* = \mm{b},
\end{align*}
where $\mm{x}$ is the unconditional sample and $\norm{\mm{y}}_Q^2 =
\mm{y}^T \mm{Qy}$.  That is, the conditional sample is the closest
vector that satisfies the linear constraint to the unconditional
sample when the distance is measured in the natural norm induced by the GMRF.

The methods in this section are all predicated on the computation of a Cholesky factorisation.  However, for genuinely large problems, it may be impossible to compute or store the Cholesky factor.  With this in mind, a suite of methods were developed by \citet{Simpson2007} based on modern iterative methods for solving sparse linear systems.  These methods have shown some promise for large problems \citep{Strickland2010,Aune2011}, although more work is needed on approximating the log density \citep{Simpsonthesis,Aune2011ISI}.

\subsection{The effect of conditioning: fast Bayesian inference} \label{sec:conditioning}

The second appealing property of GMRFs is that they behave well under conditioning.  The discussion in this section is intimately tied to discretely specified GMRFs, however we will see in Section \ref{sec:A_matrix} that the formulation below, and especially the matrix $\mm{A}$, has an important role to play in the continuous setting.   Consider the simple Bayesian hierarchical model 
\begin{subequations} \label{hierachy1}
\begin{align}
 \mm{y} | \mm{x} &\sim N( \mm{Ax}, \mm{Q}_y^{-1}) \\
 \mm{x} & \sim N(\mm{\mu}, \mm{Q}_x^{-1}),
\end{align}
\end{subequations}
where $\mm{A}$, $\mm{Q}_x$ and $\mm{Q}_y$ are sparse matrices.  A simple manipulation shows that $(\mm{x}^T, \mm{y}^T)^T$ is jointly a Gaussian Markov random field with joint precision matrix 
\begin{equation} \label{joint_prec}
\mm{Q}_{xy} = \begin{pmatrix} 
 \mm{Q}_x + \mm{A}^T\mm{Q}_y\mm{A} & - \mm{A}^T\mm{Q}_y \\
 -\mm{Q}_y\mm{A} & \mm{Q}_y
\end{pmatrix}
\end{equation}
and the mean defined implicitly through the equation $$
\mm{Q}_{xy} \mm{\mu}_{xy} = \begin{pmatrix} \mm{Q}_x \mm{\mu} \\ \mm{0} \end{pmatrix}.
$$  As $(\mm{x}^T,\mm{y}^T)^T$ is jointly a GMRF, it is easy to see \citep[see, for example,][]{book80} that 
\begin{equation} \label{x_given_y}
 \mm{x} | \mm{y} \sim N \left( \mm{\mu} +   (\mm{Q}_x +
   \mm{A}^T\mm{Q}_y\mm{A})^{-1} \mm{A}^T\mm{Q}_y (\mm{y} -\mm{A\mu}) ,  (\mm{Q}_x + \mm{A}^T\mm{Q}_y\mm{A})^{-1} \right).
\end{equation}
It is important to note that the precision matrices for the joint  \eqref{joint_prec} and the conditional \eqref{x_given_y} distributions are only sparse---and the corresponding fields are only GMRFs---if $\mm{A}$ is sparse.  This observation directly links the structure of the matrix $\mm{A}$ to the availability of efficient inference methods and will be important in the coming sections.

For practical problems in spatial statistics, the model \eqref{hierachy1} is not enough: there will be unknown parameters in both the likelihood and the latent field.  If we group the unknown parameters into a vector $\mm{\theta}$, we get the following hierarchical model
\begin{subequations} \label{hierachy2}
\begin{alignat}{2}
 &\mm{y} | \mm{x},\mm{\theta} &&\sim N( \mm{Ax}, \mm{Q}_y(\mm{\theta})^{-1}) \label{observation} \\
 &\mm{x} | \mm{\theta}  &&\sim N(\mm{\mu}, \mm{Q}_x(\mm{\theta})^{-1}) \\
&\mm{\theta} &&\sim \pi(\mm{\theta}).
\end{alignat}
\end{subequations}
In order to perform inference on \eqref{hierachy2}, it is common to use Markov chain Monte Carlo (MCMC) methods for sampling from the posterior $\pi(\mm{x},\mm{\theta} | \mm{y})$, however, this is not necessary.  It's an easy exercise in Gaussian density manipulation to show that the marginal posterior for the parameters, denoted $\pi(\mm{\theta} | \mm{y})$ can be computed without integration and is given by $$
\pi(\mm{\theta} | \mm{y}) \propto \left.\frac{ \pi(\mm{x},\mm{y}|\mm{\theta}) \pi(\mm{\theta})}{\pi(\mm{x}|\mm{y},\mm{\theta})}\right|_{\mm{x} = \mm{x}^*},
$$ where $\mm{x}^*$ can be any point, but is typically taken to be the conditional mode $E(\mm{x} | \mm{y},\mm{\theta})$,  and the corresponding marginals $\pi(\theta_j| \mm{y})$ can be computed using numerical integration.  Similarly, the marginals $\pi(x_i | \mm{y})$ can be computed using numerical integration and the observation that, for every $\mm{\theta}$, $\pi(\mm{x},\mm{y} | \mm{\theta})$ is a GMRF \citep{art375,art451}.  

It follows that, for models with Gaussian observations, it is possible to perform \emph{deterministic} inference that is exact up to the error in the numerical integration.  In particular, if there are only a moderate number of parameters, this will be extremely fast.  For non-Gaussian observation processes, exact deterministic inference is no longer possible, however, \citet{art451} showed that it is possible to construct extremely accurate approximate inference schemes by cleverly deploying a series of Laplace approximation.  The integrated nested Laplace approximation (INLA) has been used successfully on a large number of spatial problems \citep[see, for example,][]{art460,art494,Schrodle2011,Riebler2011,Cameletti2011,illianalb:11} and a user friendly \texttt{R} interface is available from \url{http://r-inla.org}.

\section{Continuously specified, Markovian Gaussian random fields} \label{sec:continuous}

One of the primary aims of spatial statistics is to infer a \emph{spatially continuous} surface $x(s)$ over the region of interest.  It is, therefore, necessary to build probability distributions over the space of functions, and the standard way of doing this is to construct Gaussian random fields, which are the generalisation to functions of multivariate Gaussian distributions in the sense that for any collection of points $(s_1,s_2,\ldots, s_p)^T$, the field evaluated at those points is jointly Gaussian.  In particular $\mm{x} \equiv (x(s_1), x(s_2), \dots, x(s_p))^T \sim N (\mm{\mu}, \mm{\Sigma})$, where the covariance matrix is given by $\Sigma_{ij} = c(s_i,s_j)$ for some positive definite covariance function $c(\cdot,\cdot)$.   In most commonly used cases, the covariance function is non-zero everywhere and, as a result, $\mm{\Sigma}$ is a dense matrix. 

It is clear that we would like to transfer some of the pleasant computational properties of GMRFs, which are outlined above, to the Gaussian random field setting.  The obvious barrier to this is that classical GMRF models are strongly tied to discrete sets of points, such as  graphs \citep{book80} and lattices \citep{art122}.  Throughout this section, we will discuss the recent work of \citet{Lindgren2011} that has broken down the barrier between GMRFs and spatially continuous Gaussian random field models.

\subsection{The spatial Markov property}

\begin{figure}[tp]
  \centering
	\includegraphics[width=0.5\textwidth]{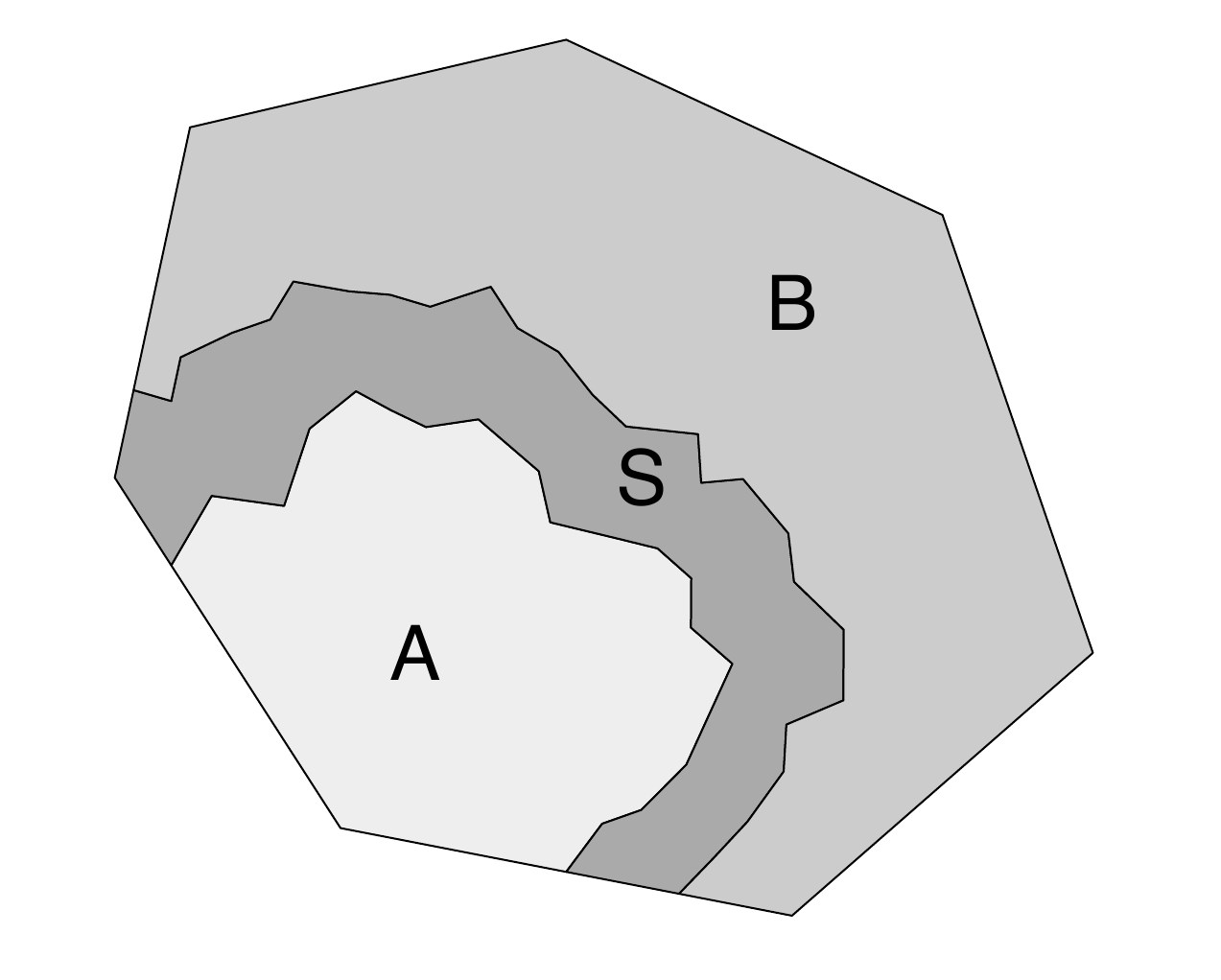}
	\caption{An illustration of the spatial Markov property.  If, for any appropriate set $S$, $\{x(s): s\in A)\}$ is independent of  $\{x(s): s\in B)\}$ given  $\{x(s): s\in S)\}$, then the field $x(s)$ has the spatial Markov property. \label{spatial_markov}}
\end{figure}

For temporal processes, defining the Markov property is greatly simplified by the structure of time: its directional nature and the clear distinction between past, present and future allow for a very natural discussion of neighbourhoods.  Unfortunately, space is far less structured and, as such the Markov property is harder to define exactly.  Intuitively, however, the definition is generalised in an the obvious way and is demonstrated by Figure \ref{spatial_markov}.  Informally, a Gaussian random field $x(s)$ has the spatial Markov property if, for every appropriate set $S$ separating $A$ and $B$,  the values of $x(s)$ in $A$ are conditionally independent of the values in $B$ given the values in $S$.   A formal definition of the spatial Markov property can be found in \citet{Rozanov1977}.

It is not immediately obvious how the spatial Markov property can be used for computational inference.  However, in an almost completely ignored paper, \citet{Rozanov1977} provided the vital characterisation of Markovian Gaussian random fields in terms of their power spectra.  The power spectrum of a stationary Gaussian random field is defined as the Fourier transform of its covariance function $c(\mm{h})$, that is $$
R(\mm{k}) \equiv \frac{1}{(2\pi)^d} \int_{\mathbb{R}^d} \exp(-i \mm{k}^T\mm{h}) c(\mm{h}).
$$  Rozanov showed that a stationary field is Markovian if and only if $R(\mm{k}) = {1}/{p(\mm{k})}$, where $p(\mm{k})$ is a  positive, symmetric polynomial.  

\subsection{From spectra to differential operators}

While Rozanov's characterisation of stationary Markovian random fields in terms of their power spectra is elegant, it is not obvious how we can turn it into a useful computational tool.  The link comes from Fourier theory: there is a one-to-one correspondence between polynomials in the frequency space and differential operators.  To see this, define the covariance operator $C$ of the Markovian Gaussian random field as the convolution 
\begin{align*}
C [f(\cdot)](\mm{h}) & = \int_{\mathbb{R}^d} c(\mm{h}' - \mm{s} ) f(\mm{h}')\,d\mm{h}' \\
 &= \int_{\mathbb{R}^d}\exp(i \mm{k}^T \mm{h}) \frac{\hat{f}(\mm{k})}{p(\mm{k})} \, d\mm{k},
\end{align*} where $p(\mm{k}) \equiv 1/R(\mm{k}) = \sum_{|\mm{i}| \leq
  \ell} a_i \mm{k}^{\mm{i}}$ is a $d$--variate, positive, symmetric
polynomial of degree $\ell$; $\mm{i} = (i_1,\ldots,i_d)^T \in
\mathbb{N}^d$ is a multi-index, meaning that $\mm{k}^{\mm{i}} =
\prod_{l=1}^d k_l^{i_l}$ and $|\mm{i}| = \sum_{l=1}^d i_l$; $f(\cdot)$ is a smooth function that goes to zero rapidly at infinity; and $\hat{f}(\mm{k})$ is the Fourier transform of $f(\mm{h})$. It follows from Fourier theory that the covariance operator $C$ has an inverse, which we will call the \emph{precision} operator $Q$, and it is given by
\begin{align*}
Q[f(\cdot)](\mm{h}) \equiv C^{-1} [f(\cdot)](\mm{h}) &= \int_{\mathbb{R}^d}\exp(i \mm{k}^T \mm{h}) \hat{f}(\mm{k})p(\mm{k}) \, d\mm{k} \\
&=\sum_{|\mm{i}| \leq \ell} a_{\mm{i}} D^{\mm{i}}f(h),
\end{align*}
where $D^{\mm{l}} = i^{|\mm{l}|}\frac{\partial^{|\mm{l}|}}{\partial h_1^{l_1}\partial h_2^{l_2}\cdot \partial h_d^{l_d}}$ are the appropriate multivariate derivatives.  

The following proposition summarises the previous discussion and specialises the result to isotropic fields.

\begin{proposition} \label{prop:markov}
A stationary Gaussian random field $x(s)$ defined on $\mathbb{R}^d$ is Markovian if and only if its covariance operator $C[f(\cdot)]$ has an inverse of the form $$
Q[f(\cdot)](\mm{h}) = \sum_{|\mm{i}| \leq \ell} a_{\mm{i}} D^{\mm{i}}f(h),
$$
where $a_{\mm{i}}$ are the coefficients of a real, symmetric polynomial $p(\mm{k}) = \sum_{|\mm{i}| \leq \ell} a_{\mm{i}} \mm{k}^{\mm{i}}$.  Furthermore, $x(s)$ is isotropic if and only if
$$
Q [f(\cdot)](\mm{h})  = \sum_{i = 0}^\ell \tilde{a}_i (-\Delta)^i f(h),
$$ where $\Delta = \sum_{i=1}^d \frac{\partial^2}{\partial h_i^2}$ is the $d$-dimensional Laplacian and $\tilde{p}(t)=\sum_{i=0}^\ell \tilde{a}_i t^i$ is a real, symmetric univariate polynomial.
\end{proposition}

The critical point of the above paragraph is that for Markovian Gaussian random fields, the covariance is the inverse of a \emph{local} operator, in the sense that the value of $Q [f(\cdot)](\mm{h})$ only depends on the value of $f(\mm{h})$ in an infinitesimal neighbourhood of $\mm{h}$.  This is in stark contrast to the covariance operator, which is an integral operator and therefore depends on the value of $f(\mm{h})$ everywhere that the covariance function is non-zero.  It is this locality that will lead us to GMRFs and the promised land of sparse precision matrices and fast computations.

\subsection{ Stochastic differential equations and Mat\'{e}rn fields}
The discussion in the previous subsection gave a very general form of Markovian Gaussian random fields.  In this section we will, for the sake of sanity, simplify the setting greatly.  Consider the stochastic partial differential (SPDE)
\begin{equation} \label{spde}
(\kappa^2 - \Delta)^{\alpha/2} x(s) = W(s),
\end{equation} where $W(s)$ is Gaussian white noise and $\alpha > 0$.  The solution to the SPDE will be a Gaussian random
field and some formal manipulations \citep[made precise by][]{Rozanov1977} show that the precision operator is
\begin{align*}
Q &= \left[ E(x x^*)\right]^{-1}= \left[(\kappa^2 - \Delta)^{-\alpha/2} E(W W^*) (\kappa^2 - \Delta)^{-\alpha/2} \right]^{-1} \\
&= (\kappa^2 - \Delta)^{\alpha}.
\end{align*} It follows that $Q$ is a local differential operator when $\alpha$ is an integer and, by Proposition \ref{prop:markov}, the stationary solution to \eqref{spde} for integer $\alpha$ is a Markovian Gaussian random field.  

Somewhat surprisingly, we have now ventured back into the more common parts of spatial statistics: \citet{art246,art455} showed that, for any $\alpha > d/2$, the solution to \eqref{spde} has a Mat\'{e}rn covariance function.  That Mat\'{e}rn family of covariance functions, which is given by $$
c_{\sigma^2,\nu,\kappa} (\mm{h}) = \frac{\sigma^2}{ 2^{\nu -1} \Gamma(\nu) } (\kappa \norm{\mm{h}})^\nu K_\nu (\kappa \norm{\mm{h}}),
$$
where $\sigma^2$ is the variance parameter, $\kappa$ controls the range, and $\nu = \alpha - d/2>0$ is the shape parameter,  is one of the most widely used families of   stationary, isotropic covariance functions.  The results of the previous section show that, when $\nu + d/2$ is an integer, the Mat\'{e}rn fields are Markovian.

\subsection{Approximating Gaussian random fields: the finite element method}

\begin{figure}[ht]
\subfigure[A continuous function]{
  \centering
	\includegraphics[width=0.45\textwidth]{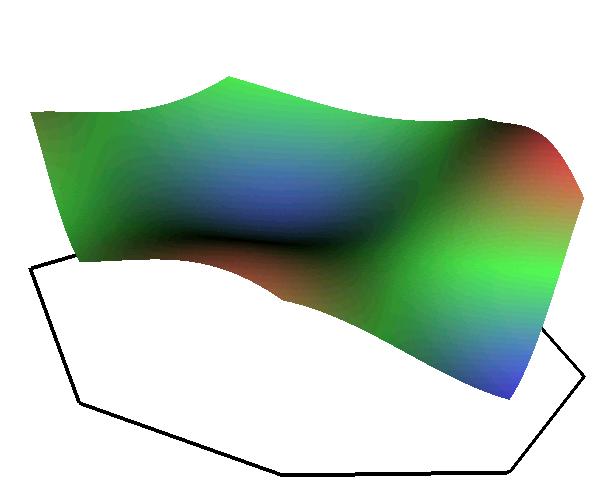}
}
\subfigure[A piecewise linear approximation]{
  \centering
	\includegraphics[width=0.45\textwidth]{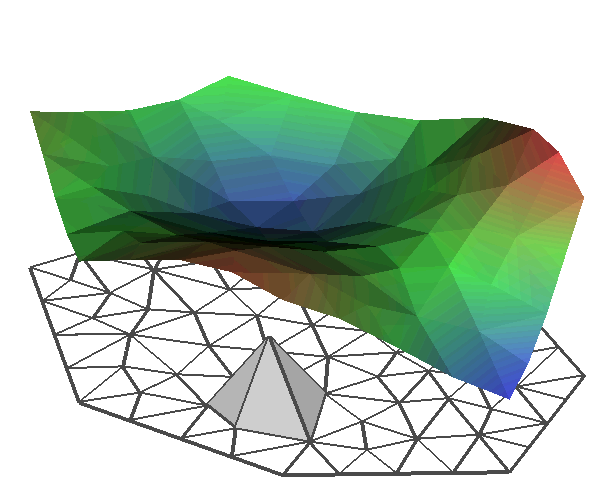}

}
\caption{Piecewise linear approximation of a function over a triangulated mesh. \label{fig:approx}}
\end{figure}

Having now laid all of the theoretical groundwork, we can consider the
fundamental question in this paper: how do can we use GMRFs to approximate a Gaussian
random field?  In order to construct a good approximation, it is vital
to remember that every realisation of a Gaussian random field is a
\emph{function}.  It follows that any sensible method for
approximating  Gaussian random fields is necessarily connected to a
method for approximating an appropriate class of deterministic functions.  Therefore, following \citet{Lindgren2011}, we are on the hunt for simple methods for approximating functions.  

A reasonably simple method for approximating continuous functions is demonstrated in Figure \ref{fig:approx}, which shows a piecewise linear approximation to a deterministic function $f(s)$ that is defined over a triangulation of the domain.  This approximation is of the form $$
f(s) \approx f_h(s) = \sum_{i=1}^n w_i \phi_i(s),
$$ where the basis function $\phi_i(s)$ is the piecewise linear function that is equal to one at the $i$th vertex of the mesh and is zero at all other vertices and the subscript $h$ denotes the largest triangle edge in the mesh and is used to differentiate between piecewise linear functions over a mesh and general functions.  The grey pyramid in Figure \ref{fig:approx} is an example of a basis function.   With this in hand, we can define a piecewise linear Gaussian random field as 
\begin{equation} \label{basis_expansion}
x_h(s) = \sum_{i=1}^n w_i \phi_i(s),
\end{equation} where $\phi_i(s)$ are as before and the weights $\mm{w}$ are now jointly a Gaussian random vector.

Our aim is now to find the statistical properties of $\mm{w}$ that make $x_h(s)$ approximate the Markovian Mat\'{e}rn fields.  In order to do this, we use the characterisation of a Mat\'{e}rn field as the stationary solution to \eqref{spde}.  For simplicity, let us consider $\alpha =2$ on a domain $\Omega \subset \mathbb{R}^2$.   Any solution of \eqref{spde} also satisfies, for any suitable function $\psi(s)$, 
\begin{align}
\int_\Omega \psi(s) (\kappa^2 - \Delta)x(s)\,ds &= \int_\Omega \psi(s) \, W(ds) \notag \\
\intertext{and an application of Green's formula leads to}
\int_\Omega \kappa^2 \psi(s) x(s) + \nabla \psi(s) \cdot \nabla x(s)\, dx &= \int_\Omega \psi(s) \, W(ds), \label{weak}
\end{align}
where the integrals on the right hand sides are integrals with respect
to white nose \citep[see Chapter  5 of][]{book104} and the second line
follows from Green's formula and the (new) condition that the normal
derivative of $x(s)$ vanishes on the boundary of $\Omega$.
Furthermore, if we find $x(s)$ such that \eqref{weak} holds for
\emph{any} sensible $\psi(s)$, then $x(s)$ is the weak solution to \eqref{spde} and it can be shown that it is a Gaussian random field with the appropriate Mat\'{e}rn covariance function. 

Unfortunately, we cannot test \eqref{weak} against every function $\psi(s)$, so we will instead chose a finite set $\{\psi_j(s)\}_{j=1}^n$ to test against.  Substituting $x_h(s)$ into \eqref{weak} and testing against this set of $\psi_j$s, we get the system of linear equations 
\begin{equation}\label{discrete}
\sum_{i=1}^n \left(\kappa^2 \int_\Omega \psi_j(s)\phi_i(s)\,ds +
  \int_\Omega \nabla\psi_j(s) \cdot \nabla \phi_i(s)\,ds \right)w_i=
\int_\Omega \psi_j(s)\,dW(s), \qquad j = 1,\dots, n.
\end{equation}  Finally, we chose our test functions $\psi_j(s)$ to be
the same as our basis functions $\phi_i(s)$ and arrive at the Galerkin
finite element method.   For piecewise linear functions over a
triangular mesh, it is easy to compute all of the integrals on the
left hand side of \eqref{discrete}.  The white noise integral on the
right hand side can be computed and it is Gaussian with mean zero and 
\begin{equation} \label{white_noise}
\operatorname{Cov}\left( \int_\Omega \phi_i(s)\,dW(s),\int_\Omega \phi_j(s)\,dW(s)\right) = \int_\Omega \phi_i(s)\phi_j(s)\,ds.
\end{equation}
We can, therefore write \eqref{discrete} in matrix form as $$
\mm{K}\tilde{\mm{w}} \sim N( \mm{0}, \tilde{\mm{C}}),
$$ where $\mm{K} = \kappa \tilde{\mm{C}} + \mm{G}$ and the matrices are given by $\tilde{C}_{ij} =  \int_\Omega \phi_i(s)\phi_j(s)\,ds$ and ${G}_{ij} =\int_\Omega \nabla\psi_j(s) \cdot \nabla \phi_i(s)\,ds $.  

A quick look at the definitions of $\tilde{\mm{C}}$  and $\mm{G}$
shows that, due to the highly local nature of the basis functions,
these matrices are sparse.  Unfortunately,  $\tilde{\mm{C}}^{-1}$ is a
dense matrix and, therefore  $\tilde{\mm{w}}$ will not be a GMRF.
However, replacing  $\tilde{\mm{C}}$
by the diagonal matrix $\mm{C} = \operatorname{diag} (\int_\Omega
\phi_i(s)\,ds, i=1,\ldots, n)$ gives essentially the same result numerically \citep{BolinWavelet} and can be shown not to increase the rate of convergence of the approximation \citep[Appendix C.5 of][]{Lindgren2011}.  Using $\mm{C}$, it follows that the solution to $$
\mm{Kw} \sim N( \mm{0}, \tilde{\mm{C}})
$$ is a GMRF with zero mean and sparse precision matrix $\mm{Q} = \mm{K}^T\mm{CK}$.  The GMRFs that correspond to other integer values of $\alpha$ can be found in \citet{Lindgren2011}.

\subsection{A continuous approximation to a continuous random field}
We have now derived a GMRF $\mm{w}$ from the continuous Mat\'{e}rn
field $x(s)$.  It is, therefore, reasonable to wonder how well
$\mm{w}$ approximates $x(s)$.  \emph{It doesn't!}  The GMRF $\mm{w}$
was defined as the weights of a basis function expansion
\eqref{basis_expansion} and only make sense in this context.  In
particular, $\mm{w}$ is not trying to approximate $x(s)$ at the mesh
vertices.  Instead, the piecewise linear Gaussian random field $x_h(s)
= \sum_{i=1}^n w_i \phi_i(s)$ tries to approximate the true random
field \emph{everywhere}.  Because of the nature of this continuous
approximation,  $x_h(s)$ will necessarily  overestimate the variance at the vertices and underestimate in in the centre of the triangles.

One of the real advantages of using piecewise linear basis functions is that a lot of work has gone into working out their approximation properties \citep{book107}.  We can leverage this information to  get  hard bounds on the convergence of $x_h(s)$ to $x(s)$.   The following theorem is a simple example of this type of result.  It shows that functionals of both the field and its derivative converge to the true functionals and the error is $\mathcal{O}(h)$, where $h$ is the length of the largest edge in the mesh.  The theorem also links the convergence of functionals of the random field to how well the piecewise linear basis functions can approximate functions in the Sobolev space $H^1$, which consists of  square integrable functions $f(s)$ for which $\norm{f}_{H^1}^2 = \kappa^2 \int_\Omega f(s)^2\,ds + \int_\Omega \nabla f(s) \cdot \nabla f(s)\,ds$ is finite.

\begin{theorem} \label{convergence}
Let $L = \kappa^2 - \Delta$.  Then, for any $f \in H^1$, $$
E \left( \int_\Omega f(s) L(x(s) - x_h(s))\,ds \right)^2 \leq c h^2 \norm{f}_{H^1}^2,
$$ where $c$ is a constant and $h$ is the size of the largest triangle edge in the mesh.
\end{theorem}

\begin{proof}
Let $f \in {H}^1$, $f_h(s)$ be the $H^1$--projection of $f$ onto the
finite element space $V_h =
\operatorname{span}\{\phi_1(s),\ldots,\phi_n(s)\}$, that is let $f_h$
be the solution to $$
\min_{f_h \in V_h} \norm{f - f_h}_{H^1}.
$$ It follows that 
\begin{align*}
 \int_\Omega f(s) L x_h(s)\,ds &=  \int_\Omega (f(s) - f_h(s)) L x_h(s)\,ds +  \int_\Omega f_h(s) L x_h(s)\,ds\\
		&= \int_\Omega f_h(s) L x_h(s)\,dx  \\
		&=\int_\Omega f_h(s)\,dW(s),
\end{align*}
where the second equality follows from the Galerkin property of $x_h(s)$ \citep{book107}, which states that  $ \int_\Omega g(s) L x_h(s)\,ds = 0$ whenever $g(s)$ is in the orthogonal complement of $V_h$ with respect to the $H^1$ inner product.  It follows directly that $$
\int_\Omega f(s) L(x(s) - x_h(s))\,ds =   \int_\Omega (f(s) - f_h(s))\,dW(s) .
$$ Therefore, it follows from the properties of white noise integrals that
\begin{align*}
E \left( \int_\Omega f(s) L(x(s) - x_h(s))\,ds \right)^2&= E \left( \int_\Omega(f(s) - f_h(s))\,dW(s) \right)^2\\
&= \int_\Omega(f(s) - f_h(s))^2\,ds.
\end{align*}
  It follows from Theorem 4.4.20 in \citet{book107} that  (under some suitable assumptions on the triangulation) $$
\norm{f_n - f}_{L^2(\Omega)} \leq c h \norm{f}_{H^1}.
$$ 
\end{proof}
The key lesson from the proof of Theorem \ref{convergence} is that the
convergence of functionals of $x_h(s)$ depends solely on how well the
basis functions can approximate a fixed $H^1$ function.  Therefore, it
is vital to consider the approximation properties of your basis functions!

\subsection{Choosing basis functions: don't forget about $\mm{A}$} \label{sec:A_matrix}

While we have computed everything with respect to piecewise linear functions, the methods considered above work for any set of test and basis functions for which all of the computations make sense.  However, we strongly warn against using any of the more esoteric choices.  There are two main issues that can appear.  The first issue is that the wrong choice of basis functions will destroy the Markov structure of the posterior model, which will annihilate the computational gains we have worked so hard for.  The second issue, which is related to Theorem \ref{convergence} is much more problematic: not all sets of basis functions will provide good approximations to $x(s)$.

In Section \ref{sec:conditioning}, we looked at the GMRF computations for hierarchical Gaussian models.  Consider a simple Gaussian observation process of the form $$
y_i \sim N( x_h(s_i), \sigma^2), \qquad i = 1,\ldots,N
$$ where the number of data points $N$ does not need to be related to the number of mesh vertices $n$.  It follows that the datapoint $(s_i,y_i)$ requires the computation of the sum $\sum_{j=1}^n w_j \phi_j(s_i)$.  When the basis functions are local, there are only a few non-zero terms in this sum and the corresponding matrix $\mm{A}$, which has entries $A_{ij} = \phi_j(s_i)$,  is sparse.  For piecewise linear functions on a triangular mesh, each row of $\mm{A}$ has at most $3$ non-zero entries.  

On the other hand, consider a basis consisting of  the first $n$ functions from the  Karhunen-Lo\'{e}ve expansion of $x(s)$.  In this case, it's easy to show that the precision matrix for $\mm{w}$ will be diagonal.  Unfortunately, these basis function are usually non-zero everywhere, so $\mm{A}$ will be a completely dense $N \times n$ matrix, which, for large datasets  will become the dominant computational cost.

\begin{figure}[tp]
  \centering
	\includegraphics[width=0.7\textwidth]{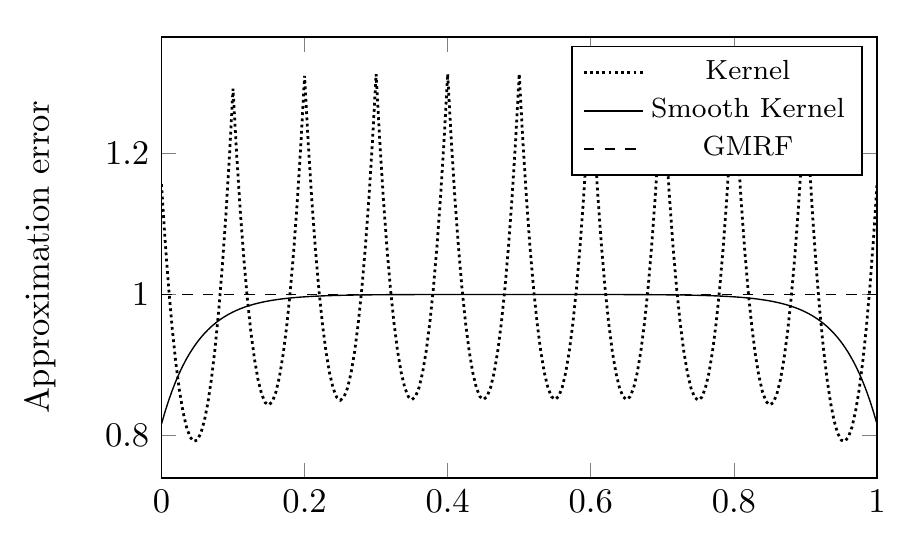}
	\caption{Taken from \citep{Simpson2011}.  The error in the best approximation to $x(s) =1$, $s\in [0,1]$. The dotted line is the na\"ive kernel basis, while the solid line is an integrated kernel basis derived by \citet{Simpson2011}.  The piecewise linear basis (dashed line) reproduces the function exactly.\label{kernel}}
\end{figure}
The approximation properties of sets of basis functions are typically
very hard to determine.  There are, however, some good guiding
principles.  The first principle is that your basis functions should
do a good job approximating simple functions.  You usually want to be
able to at least approximate constant and linear functions well.   A
second guiding principle is that you should be very careful when your
basis functions depend on a parameter that is being estimated.  It is
important to check that the approximation is still sensible over the entire parameter range.   Figure \ref{kernel} is an example of this problem, taken from \citet{Simpson2011}, shows the best approximation to a constant function using the basis derived from a convolution kernel approximation to a one dimensional Mat\'{e}rn field with $\nu = 3/2$ and $\kappa = 20$.  The basis functions are computed with respect to a fixed grid of $11$ equally spaced knots $s_i$, $i=1,\dots,11$ and are given by $$
\phi_i(s) = \frac{1}{(2\pi)^{1/2}\kappa}\exp(-\kappa|s - s_i|).
$$ When $\kappa$ gets large, the basis functions get sharper and the
approximation gets worse.  At its worst, the effective range of the
kernel becomes shorter than grid of knots.  Therefore, if $\kappa$
becomes large relative to the grid spacing, both Kriging estimates and
variance estimates will be badly infected by the poor choice of basis
function.  On the other hand, it can be shown that piecewise linear
Galerkin finite element approximations are stable in the sense that
the norm of the approximation can be bounded above by an appropriate
norm of the true solution.  This means that a piecewise linear basis, when used in this way, 
will \emph{never} display this bad behaviour.

Another consideration when choosing sets of basis functions is the smoothness of the random field.  Theorem \ref{convergence} directly links the convergence of the approximate fields to how well the basis functions can approximate a certain class of functions.  This will usually be the case.  Typically, the smoother the field is, the more useful higher order ``spectral'' bases will be.  For a smooth enough field, it may actually be cheaper to use a small, well chosen global basis  than a large basis full of  local functions.

\subsection{Extending the models: anisotropy, drift, and space-time models on manifolds}

\begin{figure}[tp]
\subfigure[A sample from an anisotropic field]{
  \centering
  \includegraphics[width=0.45\textwidth,angle=0]{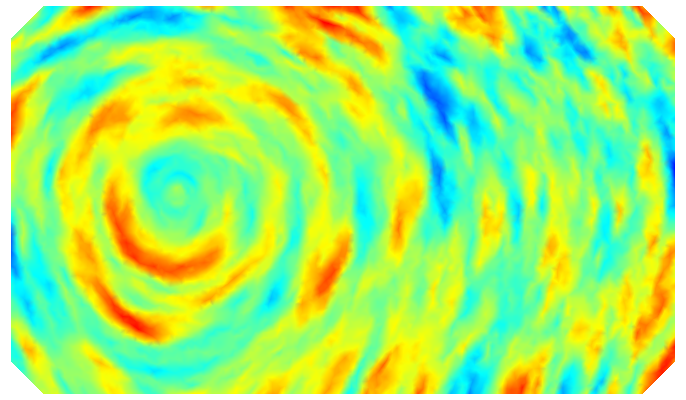}
}
\subfigure[The covariance functions]{
  \centering
	\includegraphics[width=0.45\textwidth,angle=0]{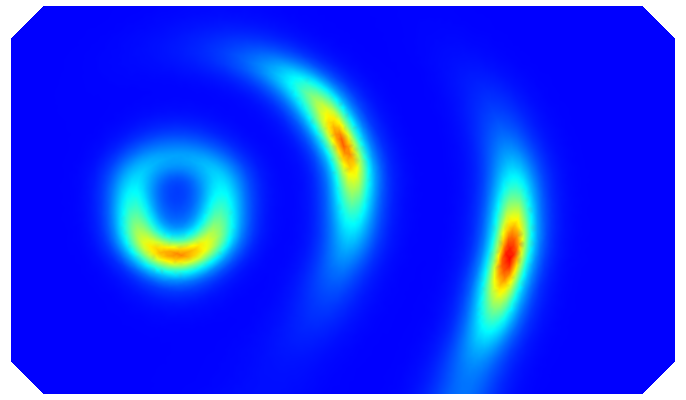}
}

	\caption{An anisotropic field defined through \eqref{spde}
          when the Laplacian is replace with the anisotropic variant $\nabla
    \cdot(\mm{D}(s) \nabla (\tau x(s)))$.  The right hand figure shows
    the approximate covariance function evaluated at three points and
    the non-stationarity is evident.\label{anisotropic}}
\end{figure}

One of the main advantages to the SPDE formulation is that it is easy
to construct non-stationary variants. Non-stationary fields can now be
defined by letting the parameters in~\eqref{spde}, be
space-dependent. For example, $\log\kappa$ can be expanded using a few 
weighted smooth basis functions
\begin{equation}\label{eq8}%
    \log \kappa(\mm{s}) = \sum_{i} \beta_{i} b_{i}(\mm{s})
\end{equation}    
and similar expansions can be used for $\tau$.  This extension requires only minimal
changes to the method used in the stationary case.  More
interestingly, we can also incorporate models of spatially varying
anisotropy by replacing the Laplacian $\Delta$ in \eqref{spde} with
the more general operator
\begin{displaymath}
    \nabla
    \cdot(\mm{D}(s) \nabla (\tau x(s))) \equiv \sum_{i,j=1}^2
    \frac{\partial}{\partial s_i} \left( D_{ij}(s) \frac{\partial}{\partial
          s_j} (\tau(s) x(s))\right).
\end{displaymath}
These models can be extended even farther by incorporating a ``drift''
term, as well as temporal dependence, which leads to the general model
\begin{equation}\label{eq1}%
    \frac{\partial}{\partial t} (\tau
    u(s,t)) + (\kappa^2(\tau x(s,t)) - \nabla \cdot(\mm{D}\nabla (\tau
    x(s,t)))) + \mm{b} \cdot \nabla (\tau x(s,t)) = W(s,t),
\end{equation}
where $t$ is the time variable and $\mm{b}$ is a vector describing the
direction of the drift and the dependence of $\kappa$, $\tau$,
$\mm{D}$ and $\mm{b}$ on $s$ and $t$ has been suppressed.  The concepts behind the construction of Markovian approximations to the stationary SPDE transfer almost completely to this general setting, however more care needs to be taken with the exact form of the discretisation \citep{FuglstadSpaceTime}.

A strong advantage of defining non-stationarity through \eqref{eq1} is that the underlying physics  is well understood.  For example the second order term $\nabla \cdot(\mm{D}\nabla (\tau
    x(s,t))))$ describes the diffusion, where the matrix $\mm{D}(s,t)$
    describes the preferred direction of diffusion at location $s$ and
    time $t$. A spatial random field with a non-constant diffusion
    matrix $\mm{D}(s)$ is shown  in Figure \ref{anisotropic}.
    Similarly, the first order term $\mm{b} \cdot \nabla (\tau
    x(s,t))$ describes an advection effect and $\mm{b}(s,t)$ can be
    interpreted as the direction of the forcing field.  We note that
    unlike other methods for modelling non-stationarity, the
    parametrisation  in this model is \emph{unconstrained}.  This is
    in contrast to the deformation methods of~\citet{art244}, in which
    the deformation must be constrained to be one-to-one.  Initial
    work on parameterising and performing inference on these models
    has been undertaken by \citet{FuglstadAnisotropic}.

\begin{figure}[tp]
  \centering
	\includegraphics[width=0.7\textwidth]{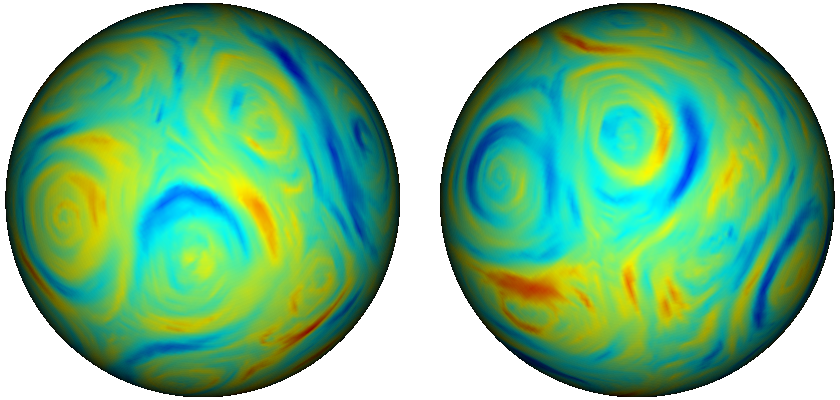}
	\caption{A sample from a non-stationary, anisotropic random
          field on the sphere.\label{sphere}}
\end{figure}

An interesting  consequence of defining our models through \emph{local} stochastic partial differential equations, such as \eqref{eq1}, is that the SPDEs still make sense when $\mathbb{R}^d$ is replaced by a space that is only locally flat.   We can, therefore, use  \eqref{eq1} to \emph{define}
non-stationary Gaussian fields on manifolds, and still obtain a GMRF
representation.  Furthermore, the computations can be done in essentially the same way,
the only change is that the Gaussian noise is now a Gaussian random
measure, and that we need to take into account the local curvature
when computing the integrals to obtain the solution.
Figure~\ref{sphere} shows a realisation of a
non-stationary Gaussian random field on a sphere, with a model similar
to the one used in Figure~\ref{anisotropic}. The solution is
still explicit, so all elements of the precision matrix, for a fixed
triangularisation, can be directly computed with no extra cost.  The
ability to construct computationally efficient representations of
non-stationary random fields on a manifold is important, for example,
when modelling environmental data on the sphere.

\subsection{Further extensions: areal data, multivariate random fields and non-integer $\alpha$s}

The combination of a continuously specified approximate random field and a Markovian structure allows for the straightforward assimilation of areal and point-referenced data.  The trick is to utilise the matrix $\mm{A}$ in the observation equation \eqref{observation}.  The link between the latent GMRF and the point-referenced data is as described above.  For the areal data, the dependence is through integrals of the random field and these integrals can be computed exactly for $x_h(s)$ and the resulting sums go into $\mm{A}$.  If a more complicated  functional of the random field has been observed, this can be approximated using numerical quadrature.  This has been recently applied by \citet{SimpsonLGCP} to inferring log-Gaussian point processes. 

It is also possible to extend the SPDE framework considered above to a multivariate setting.  Primarily, the idea is to replace a single SPDE with a system of SPDEs.  It can be shown (work in progress) that these models overlap  with, but are not equivalent to, the multivariate Mat\'{e}rn models constructed by \citet{gneiting2010matern}.

The only Markovian Mat\'{e}rn models are those where the smoothing parameter is $d/2$ less than an integer.  It turns out, however, that  we can construct good Markov random field approximations when $\alpha$ isn't an integer \citep[see the authors' response in][]{Lindgren2011}.  This is essentially a powerful extension of the GMRF approximations of \citet{art197}.  We are currently working on approximating non-Mat\'{e}rn random fields using this method.

\section{Conclusion}
In this paper we have surveyed the deep and fascinating link between
the continuous Markov property and computationally efficient
inference.  The material in this paper barely scratches the surface
of the questions, applications and oddities of these fields (for
instance, \citet{BolinNonGaussian}  uses a similar construction for non-Gaussian noise processes).  However, we have demonstrated that by carefully considering the Markov property, we are able to greatly boost our modelling capabilities while at the same time ensuring that the models are fast to compute with.
The combination of flexible modelling and fast computations ensures that investigations into Markovian random fields will continue to produce interesting and useful results well into the future.

\bibliographystyle{abbrvnat}
{\bibliography{merged}}
\end{document}